\documentclass[twocolumn,aps,pra,floatfix,amsmath,amssymb,longbibliography]{revtex4-2}

\usepackage[T1]{fontenc}
\usepackage{amsthm}
\usepackage{booktabs}
\usepackage{xcolor}
\usepackage{tikz}
\usetikzlibrary{calc}
\usepackage[
  breaklinks=true,
  colorlinks=true,
  citecolor=blue,
  linkcolor=blue,
  urlcolor=blue
]{hyperref}
\usepackage{orcidlink}

\newtheorem{theorem}{Theorem}
\newtheorem{lemma}[theorem]{Lemma}

\newtheorem{corollary}[theorem]{Corollary}
\theoremstyle{definition}
\newtheorem{definition}[theorem]{Definition}

\theoremstyle{remark}

\newcommand{\ii}{\mathrm{i}}
\newcommand{\C}{\mathbb C}
\newcommand{\R}{\mathbb R}
\newcommand{\Htwenty}{\mathsf H_{20}}
\newcommand{\Htwentyone}{\mathsf H_{21}}

\definecolor{ctx1}{RGB}{230,25,75}
\definecolor{ctx2}{RGB}{60,180,75}
\definecolor{ctx3}{RGB}{0,130,200}
\definecolor{ctx4}{RGB}{245,130,48}
\definecolor{ctx5}{RGB}{145,30,180}
\definecolor{ctx6}{RGB}{70,240,240}
\definecolor{ctx7}{RGB}{240,50,230}
\definecolor{ctx8}{RGB}{210,245,60}
\definecolor{ctx9}{RGB}{250,190,190}
\definecolor{ctx10}{RGB}{0,128,128}
\definecolor{ctx11}{RGB}{128,128,0}
\definecolor{ctx12}{RGB}{170,110,40}
\definecolor{ctx13}{RGB}{128,0,0}
\definecolor{ctx14}{RGB}{0,0,128}
\tikzset{ctxline/.style={line width=0.95pt,line cap=round}}

\newcommand{\twodotlabel}[5]{%
  \node[circle,fill=#1,draw=#1,minimum size=7.2pt,inner sep=0pt,
        label={[font=\scriptsize,label distance=0.7mm]#3:{#4}}] at #5 {};
  \node[circle,fill=#2,draw=#2,minimum size=4.2pt,inner sep=0pt] at #5 {};
}

\begin{document}

\title{Complex coordinates forced by finite orthogonality hypergraphs in dimension three}

\author{Mirko Navara\,\orcidlink{0000-0002-0880-5992}}
\affiliation{Faculty of Electrical Engineering, Czech Technical University in Prague,
Technick\'{a} 2, CZ-166~27 Prague 6, Czech Republic}
\email{navara@fel.cvut.cz}

\author{Karl Svozil\,\orcidlink{0000-0001-6554-2802}}
\affiliation{Institute for Theoretical Physics, TU Wien,
Wiedner Hauptstra{\ss}e 8-10/136, A-1040 Vienna, Austria}
\email{karl.svozil@tuwien.ac.at}

\date{\today}

\begin{abstract}
We exhibit two finite orthogonality hypergraphs that admit faithful orthogonal representations by rays in complex three-dimensional space but admit none in real three-dimensional space. Thus, although each orthogonality diagram is unchanged, its coordinatizability in dimension three depends on whether the scalar field is complex or real. The first hypergraph has 20 vertices and 12 contexts. The second hypergraph adds one vertex and two contexts. Phase-adjusted realification also gives explicit faithful orthogonal representations of both hypergraphs in real six-dimensional space.
\end{abstract}

\maketitle

\section{Introduction}

An orthogonality hypergraph records which finite families of rays are required
to form orthogonal bases.  Its representability may depend on the scalar field,
even when the ambient dimension is fixed.  Harding and Salinas Schmeis gave the
first finite hypergraph with a faithful orthogonal representation over $\C^3$
but none over $\R^3$~\cite{harding2025remarksIJTP}.  Independently, Trandafir
and Cabello encountered the same real-complex separation in a
Kochen--Specker (KS) construction designed for bipartite perfect quantum
strategies~\cite{2024-cabello-trandafir0}.  A subsequent analysis by Navara and
Svozil relates and extends those constructions through mutually unbiased bases
and KS closures~\cite{svozil-2025-MYOCHS}.

The role of complex scalars in quantum theory involves two questions that
should not be conflated.  If a change of dimension and an additional complex
structure are allowed, a complex Hilbert space $\C^n$ can be represented on
the real space $\R^{2n}$ by separating real and imaginary parts; related
constructions reproduce broad classes of complex quantum protocols within a
real formalism~\cite{Hardy2012LimitedHolism,Aleksandrova2013UniversalQubit,
McKague2009SimulatingRealHilbert}.  If instead one keeps fixed the dimension,
composition rules, locality assumptions, or the interpretation of a context
as a maximal orthogonal family, real and complex models can be separated
operationally~\cite{renou-2021,Weilenmann2025PartialIndependencePRL,
Bednorz2022OptimalDiscrimination}.  A statement that a hypergraph is
representable in $\C^3$ but not in $\R^3$ belongs to the latter,
fixed-dimension comparison.

The phase-adjusted realification method of Khrennikov and
Svozil~\cite{svozil-2025-rvc} makes the boundary precise for finite
orthogonality data.  It preserves and reflects all pairwise orthogonalities in
$\R^6$, but a three-ray context then spans only half of the ambient real
space.  Below we apply this construction explicitly to both octagons after
proving their same-dimension separation.

We present two further exact finite examples, both with an octagonal incidence
skeleton.  The 20-vertex hypergraph $\Htwenty$ forces an equality of
two sums of rank-one projectors.  Complex phases permit two distinct
nonorthogonal two-ray decompositions of the same positive operator; over the
reals, the spectral decomposition fixes the rays.  The 21-vertex hypergraph
$\Htwentyone$ is obtained by completing both inner pairs with a common ray.
The spectral argument then becomes degenerate, but the perimeter closure
relations supply a second obstruction to real coordinates.

Unlike a conventional Kochen--Specker set, neither hypergraph is used through
an uncolorability argument.  Its role here is instead to enforce exact
orthogonality incidences and the resulting projector or closure identities;
no contradiction involving two-valued states is assumed or needed.

The results established below go beyond merely exhibiting convenient complex
coordinates.  For each displayed representation we verify all unordered ray
pairs exactly: the prescribed pairs are orthogonal, every other inner product
is nonzero, and distinct vertices label distinct rays.  Thus the word
``faithful'' carries its full graph-theoretic meaning~\cite{lovasz-79,Portillo-2015}.

\section{Orthogonality representations and two algebraic lemmas}

Throughout, we use the physics convention for the Hermitian inner product,
which is conjugate-linear in the first argument:
\begin{equation}
 \langle x,y\rangle=\sum_{k=1}^{d}\overline{x_k}\,y_k.
 \label{eq:inner-product}
\end{equation}

\begin{definition}
Let $\mathsf H=(V,\mathcal E)$ be a three-uniform hypergraph, let
$\mathbb F$ be either $\R$ or $\C$, and let $d\ge3$.  A \emph{faithful
orthogonal representation} of $\mathsf H$ in $\mathbb F^d$ assigns a nonzero
vector $x_v\in\mathbb F^d$ to every $v\in V$ such that:
\begin{enumerate}
\item distinct vertices determine distinct rays; and
\item $\langle x_u,x_v\rangle=0$ if and only if $u$ and $v$ belong to a
common hyperedge.
\end{enumerate}
Each hyperedge is then an orthogonal triple.  For $d=3$ it is an orthogonal
basis and hence a maximal \emph{context}; for $d>3$ the same specified triple
need not be maximal.
\end{definition}

For a nonzero column vector $x$, let $x^\dagger$ denote its conjugate
transpose and write
\begin{equation}
 P_x=\frac{xx^\dagger}{\langle x,x\rangle}
 \label{eq:projector}
\end{equation}
for the rank-one projector onto its ray.  In a three-dimensional
representation, a context $\{x,y,z\}$ satisfies
\begin{equation}
 P_x+P_y+P_z=I_3.
 \label{eq:context-sum}
\end{equation}
Consequently, integer linear combinations of three-dimensional context
equations can force projector identities that hold in every representation
of the hypergraph in that dimension.

For $x,y\in\C^3$ define the Hermitian cross product by
\begin{equation}
 x\boxtimes y:=\overline{x\times y},
 \label{eq:hcross}
\end{equation}
where $x\times y$ is the ordinary bilinear cross product.  Then
$x\boxtimes y$ is Hermitian-orthogonal to both $x$ and $y$.

\begin{lemma}[Cyclic Gram closure]
\label{lem:gram-closure}
Let $x,y,z\in\C^3$, and suppose the two cross products in
Eq.~\eqref{eq:hcross} are nonzero.  Then
\begin{equation}
 \langle x\boxtimes y,y\boxtimes z\rangle=0
 \quad\Longleftrightarrow\quad
 \langle x,y\rangle\langle y,z\rangle
 =\langle x,z\rangle\langle y,y\rangle.
 \label{eq:gram-closure}
\end{equation}
\end{lemma}

\begin{proof}
The bilinear Lagrange identity, with $\mathbin{\cdot}$ denoting the standard
bilinear dot product on $\C^3$, gives
\[
 (x\times y)\mathbin{\cdot}(\bar y\times\bar z)
 =(x\mathbin{\cdot}\bar y)(y\mathbin{\cdot}\bar z)
 -(x\mathbin{\cdot}\bar z)(y\mathbin{\cdot}\bar y).
\]
Indeed, for arbitrary $a,b\in\C^3$,
$\langle\bar a,\bar b\rangle=a\mathbin{\cdot}\bar b$; taking
$a=x\times y$ and $b=y\times z$ therefore identifies the left-hand side
with $\langle x\boxtimes y,y\boxtimes z\rangle$.  The right-hand side
is the complex conjugate of
\[
 \langle x,y\rangle\langle y,z\rangle
 -\langle x,z\rangle\langle y,y\rangle.
\]
Their vanishing is therefore equivalent.
\end{proof}

\begin{lemma}[Rigidity of two real rays]
\label{lem:real-rigidity}
Let $a_1,a_2,b_1,b_2$ be real unit vectors satisfying
\begin{equation}
 P_{a_1}+P_{a_2}=P_{b_1}+P_{b_2}.
 \label{eq:two-sums}
\end{equation}
If $a_1$ and $a_2$ are distinct and nonorthogonal, then
$\{[b_1],[b_2]\}=\{[a_1],[a_2]\}$ as unordered pairs of rays.
\end{lemma}

\begin{proof}
All four vectors lie in the range of the common operator in
Eq.~\eqref{eq:two-sums}.  Choose signs so that
$\gamma=\langle a_1,a_2\rangle\in(0,1)$.  The two nonzero eigenvalues
are $1+\gamma$ and $1-\gamma$, with eigenvectors $a_1+a_2$ and
$a_1-a_2$, respectively.  Choose signs and order for the $b$-vectors so
that $\delta=\langle b_1,b_2\rangle\ge0$.  Equality of spectra gives
$\delta=\gamma$.  Since the two eigenvalues are distinct, there are real
nonzero scalars $\lambda,\mu$ such that
\[
 b_1+b_2=\lambda(a_1+a_2),\qquad
 b_1-b_2=\mu(a_1-a_2).
\]
The unit-norm and inner-product equations give
\begin{align*}
 \lambda^2(1+\gamma)+\mu^2(1-\gamma)&=2,\\
 \lambda^2(1+\gamma)-\mu^2(1-\gamma)&=2\gamma.
\end{align*}
Hence $\lambda^2=\mu^2=1$.  Solving for $b_1,b_2$ recovers the two
original rays, possibly with signs and permutation.
\end{proof}

\section{The 20-ray octagon}
\label{sec:h20}

The vertices of $\Htwenty$ are eight midpoint vertices
$v_0,\ldots,v_7$, eight corner vertices
$v_{01},v_{12},\ldots,v_{70}$, and four inner vertices
$a_1,a_2,b_1,b_2$.  Indices are read modulo eight.  The eight perimeter
contexts and four inner contexts are
\begin{align}
 E_j&=\{v_{j-1,j},v_j,v_{j,j+1}\}, &&j=0,\ldots,7,
 \label{eq:perimeter-contexts}\\
 D_0&=\{v_0,a_1,v_4\},&D_2&=\{v_2,a_2,v_6\},\nonumber\\
 D_1&=\{v_1,b_2,v_5\},&D_3&=\{v_3,b_1,v_7\}.
 \label{eq:inner-contexts}
\end{align}
The incidence structure is shown in Fig.~\ref{fig:h20}.

\begin{figure*}[t]
\centering
\begin{tikzpicture}[scale=1.6,every label/.style={font=\scriptsize}]
  \def\Rmid{2}
  \def\Rcor{2.1648}
  \def\Rinn{1.5}
  \foreach \i/\j/\ang in {
    0/1/22.5,1/2/67.5,2/3/112.5,3/4/157.5,
    4/5/202.5,5/6/247.5,6/7/292.5,7/0/337.5
  }{\coordinate (v\i\j) at (\ang:\Rcor);}
  \foreach \i/\ang in {
    0/0,1/45,2/90,3/135,4/180,5/225,6/270,7/315
  }{\coordinate (v\i) at (\ang:\Rmid);}
  \coordinate (a1) at (180:\Rinn);
  \coordinate (a2) at (90:\Rinn);
  \coordinate (b1) at (135:\Rinn);
  \coordinate (b2) at (45:\Rinn);

  \draw[ctxline,draw=ctx1] (v70)--(v01);
  \draw[ctxline,draw=ctx2] (v01)--(v12);
  \draw[ctxline,draw=ctx3] (v12)--(v23);
  \draw[ctxline,draw=ctx4] (v23)--(v34);
  \draw[ctxline,draw=ctx5] (v34)--(v45);
  \draw[ctxline,draw=ctx6] (v45)--(v56);
  \draw[ctxline,draw=ctx7] (v56)--(v67);
  \draw[ctxline,draw=ctx8] (v67)--(v70);
  \draw[ctxline,draw=ctx9] (v0)--(v4);
  \draw[ctxline,draw=ctx10] (v1)--(v5);
  \draw[ctxline,draw=ctx11] (v2)--(v6);
  \draw[ctxline,draw=ctx12] (v3)--(v7);

  \twodotlabel{ctx1}{ctx2}{above right}{$v_{01}$}{(v01)}
  \twodotlabel{ctx2}{ctx3}{above}{$v_{12}$}{(v12)}
  \twodotlabel{ctx3}{ctx4}{above left}{$v_{23}$}{(v23)}
  \twodotlabel{ctx4}{ctx5}{left}{$v_{34}$}{(v34)}
  \twodotlabel{ctx5}{ctx6}{below left}{$v_{45}$}{(v45)}
  \twodotlabel{ctx6}{ctx7}{below}{$v_{56}$}{(v56)}
  \twodotlabel{ctx7}{ctx8}{below right}{$v_{67}$}{(v67)}
  \twodotlabel{ctx8}{ctx1}{right}{$v_{70}$}{(v70)}
  \twodotlabel{ctx1}{ctx9}{right}{$v_0$}{(v0)}
  \twodotlabel{ctx2}{ctx10}{above right}{$v_1$}{(v1)}
  \twodotlabel{ctx3}{ctx11}{above}{$v_2$}{(v2)}
  \twodotlabel{ctx4}{ctx12}{above left}{$v_3$}{(v3)}
  \twodotlabel{ctx5}{ctx9}{left}{$v_4$}{(v4)}
  \twodotlabel{ctx6}{ctx10}{below left}{$v_5$}{(v5)}
  \twodotlabel{ctx7}{ctx11}{below}{$v_6$}{(v6)}
  \twodotlabel{ctx8}{ctx12}{below right}{$v_7$}{(v7)}
  \node[circle,fill=ctx9,draw=ctx9,minimum size=7.2pt,inner sep=0pt,
        label={[font=\scriptsize]above:{$a_1$}}] at (a1) {};
  \node[circle,fill=ctx11,draw=ctx11,minimum size=7.2pt,inner sep=0pt,
        label={[font=\scriptsize]right:{$a_2$}}] at (a2) {};
  \node[circle,fill=ctx12,draw=ctx12,minimum size=7.2pt,inner sep=0pt,
        label={[font=\scriptsize]above:{$b_1$}}] at (b1) {};
  \node[circle,fill=ctx10,draw=ctx10,minimum size=7.2pt,inner sep=0pt,
        label={[font=\scriptsize]above:{$b_2$}}] at (b2) {};
\end{tikzpicture}
\caption{The 20-ray hypergraph $\Htwenty$.  The eight perimeter contexts
$E_0,\ldots,E_7$ form the octagon; the four inner contexts join opposite
midpoints through $a_1,a_2,b_1,b_2$.  Each color denotes one context.}
\label{fig:h20}
\end{figure*}
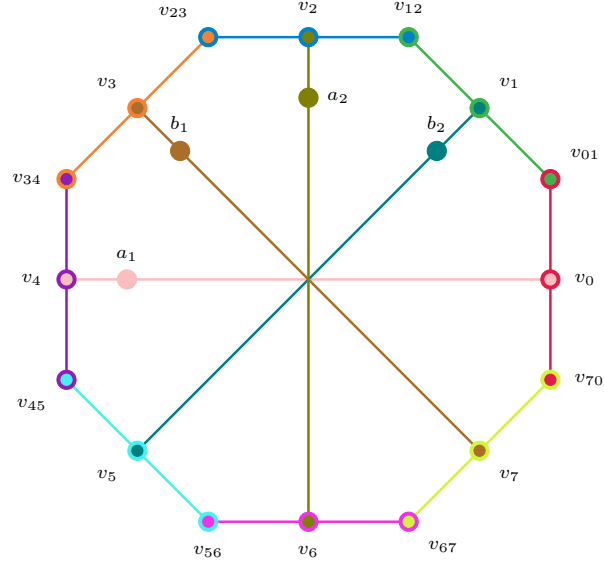

\subsection{The covering identity}

Consider the two collections of six contexts
\begin{align*}
 \mathcal A&=\{D_0,D_2,E_1,E_3,E_5,E_7\},\\
 \mathcal B&=\{D_1,D_3,E_0,E_2,E_4,E_6\}.
\end{align*}
Every midpoint and every corner occurs exactly once in each collection.
Only the two $a$-vertices occur in $\mathcal A$, and only the two
$b$-vertices occur in $\mathcal B$.  Summing Eq.~\eqref{eq:context-sum}
over the two collections and subtracting therefore gives
\begin{equation}
 P_{a_1}+P_{a_2}=P_{b_1}+P_{b_2}.
 \label{eq:covering-identity}
\end{equation}
This identity is combinatorial: it holds in every real or complex
three-dimensional representation of $\Htwenty$.

\subsection{An exact faithful representation over
\texorpdfstring{$\C^3$}{C3}}

Set
\begin{equation}
 \begin{aligned}
 C&=\frac{49-5\sqrt{69}}{26},&s&=\sqrt C,\\
 N&=1+C,&r&=\sqrt N.
 \end{aligned}
 \label{eq:C-definition}
\end{equation}
The number $C$ is the smaller positive root of
\begin{equation}
 13C^2-49C+13=0,
 \label{eq:C-polynomial}
\end{equation}
so $0<C<1$.  Define the unit phases
\begin{align}
 \eta&=\frac{5+12\ii}{13},\label{eq:eta}\\
 \zeta&=\frac{14+6\sqrt{69}}{65}
 +\ii\frac{4\sqrt{69}-21}{65}.
 \label{eq:zeta}
\end{align}
Let
\begin{align*}
 (\mu_0,\ldots,\mu_7)&=(-1,\eta,1,-\eta,-1,\eta,1,-\eta),\\
 (\tau_0,\ldots,\tau_7)&=(1,\zeta,-\ii,-\ii\zeta,
 -1,-\zeta,\ii,\ii\zeta).
\end{align*}
The inner and midpoint rays are represented by
\begin{align}
 a_1&=(s,1,0),&a_2&=(s,-1,0),\nonumber\\
 b_1&=(s,\eta,0),&b_2&=(s,-\eta,0),
 \label{eq:h20-inner-vectors}\\
 v_j&=(1,s\mu_j,r\tau_j),&&j=0,\ldots,7.
 \label{eq:h20-midpoints}
\end{align}
For $j$ modulo eight, define
\begin{equation}
 v_{j,j+1}=v_j\boxtimes v_{j+1}.
 \label{eq:h20-corners}
\end{equation}
Equivalently, Eq.~\eqref{eq:h20-corners} is the explicit formula
\begin{multline}
 v_{j,j+1}=\bigl(
 sr(\bar\mu_j\bar\tau_{j+1}-\bar\tau_j\bar\mu_{j+1}),\\
 r(\bar\tau_j-\bar\tau_{j+1}),
 s(\bar\mu_{j+1}-\bar\mu_j)\bigr).
 \label{eq:h20-corner-formula}
\end{multline}

\begin{theorem}
\label{thm:h20-complex}
Equations~\eqref{eq:C-definition}--\eqref{eq:h20-corner-formula}
give a faithful orthogonal representation of $\Htwenty$ in $\C^3$.
\end{theorem}

\begin{proof}
First, $|\eta|=1$.  Write $\zeta=x+\ii y$.  The four inner contexts
follow directly from $r^2=N$, $|\eta|=|\zeta|=1$, and the antipodal
relations $\mu_{j+4}=\mu_j$, $\tau_{j+4}=-\tau_j$.  For example,
\[
 \langle a_1,v_0\rangle=0,\qquad
 \langle v_0,v_4\rangle=1+C-r^2=0,
\]
and
\[
 \langle b_2,v_1\rangle=0,\qquad
 \langle v_1,v_5\rangle=1+C-r^2|\zeta|^2=0.
\]
The other two inner contexts are identical calculations.

By construction, $v_{j,j+1}$ is orthogonal to $v_j$ and $v_{j+1}$.
It remains to prove orthogonality of the two corner rays in every
perimeter context.  Put
\[
 G_{jk}=\langle v_j,v_k\rangle
 =1+C\bar\mu_j\mu_k+N\bar\tau_j\tau_k.
\]
Lemma~\ref{lem:gram-closure} reduces the eight remaining conditions to
\begin{equation}
 G_{j-1,j}G_{j,j+1}=G_{j-1,j+1}G_{j,j},
 \qquad j=0,\ldots,7.
 \label{eq:h20-closure}
\end{equation}
Substituting the cyclic patterns for $\mu_j$ and $\tau_j$ reduces all eight
conditions to common real and imaginary equations once $x^2+y^2=1$.
Separating those parts gives
\begin{align}
 7Cx-17Cy-13C-13x+13y+13&=0,
 \label{eq:closure-real}
\end{align}
and
\begin{multline}
 -7C^2x+17C^2y+13C^2-20Cx+30Cy\\
 {}+2C-13x+13y+13=0.
 \label{eq:closure-imag}
\end{multline}
Solving these two linear equations for $x$ and $y$ gives
\begin{equation}
 x=\frac{104-76C}{65(1+C)},\qquad
 y=\frac{39-81C}{65(1+C)}.
 \label{eq:xy-C}
\end{equation}
For these values,
\begin{equation}
 x^2+y^2-1=
 \frac{48(13C^2-49C+13)}{325(1+C)^2}=0.
 \label{eq:zeta-unit}
\end{equation}
Using the smaller root in Eq.~\eqref{eq:C-polynomial},
Eq.~\eqref{eq:xy-C} is exactly the phase in Eq.~\eqref{eq:zeta}.
Thus all twelve triples are contexts.

It remains to exclude unintended orthogonalities and ray coincidences.
The exact overlap certificate in Appendix~\ref{app:faithfulness} shows
that the 36 prescribed unordered pairs have zero overlap, whereas the
154 unprescribed pairs have squared normalized overlaps between $1/25$
and $4/5$.  Hence every unprescribed inner product is nonzero and no two
displayed vectors are proportional.  The representation is faithful.
\end{proof}

The two inner pairs indeed decompose the same operator.  Since all four
vectors have squared norm $N$ and $|\eta|=1$,
\begin{equation}
 P_{a_1}+P_{a_2}=P_{b_1}+P_{b_2}
 =\frac{2}{N}\operatorname{diag}(C,1,0).
 \label{eq:complex-two-sum}
\end{equation}
Their within-pair inner products are $C-1\ne0$, and the four rays are
distinct because $\eta\notin\{1,-1\}$.

\subsection{Nonexistence over \texorpdfstring{$\R^3$}{R3}}

\begin{theorem}
\label{thm:h20-real}
The hypergraph $\Htwenty$ has no faithful orthogonal representation in
$\R^3$.
\end{theorem}

\begin{proof}
Suppose a faithful real representation existed.  The covering argument
is independent of the field, so Eq.~\eqref{eq:covering-identity} would
hold.  Faithfulness implies that $a_1,a_2$ are distinct and
nonorthogonal, because they do not share a context; the same is true of
$b_1,b_2$.  Lemma~\ref{lem:real-rigidity} then identifies the unordered
pair of $b$-rays with the unordered pair of $a$-rays.  This contradicts
the requirement that four distinct vertices label four distinct rays.
\end{proof}

\section{The 21-ray orthogonal augmentation}
\label{sec:h21}

Add one vertex $c$ and two contexts
\begin{equation}
 F_a=\{a_1,c,a_2\},\qquad F_b=\{b_1,c,b_2\}
 \label{eq:added-contexts}
\end{equation}
to $\Htwenty$.  The result is the 21-vertex, 14-context hypergraph
$\Htwentyone$ shown in Fig.~\ref{fig:h21}.  The added contexts force both
inner pairs to be orthogonal and to have the same one-dimensional
orthogonal complement.

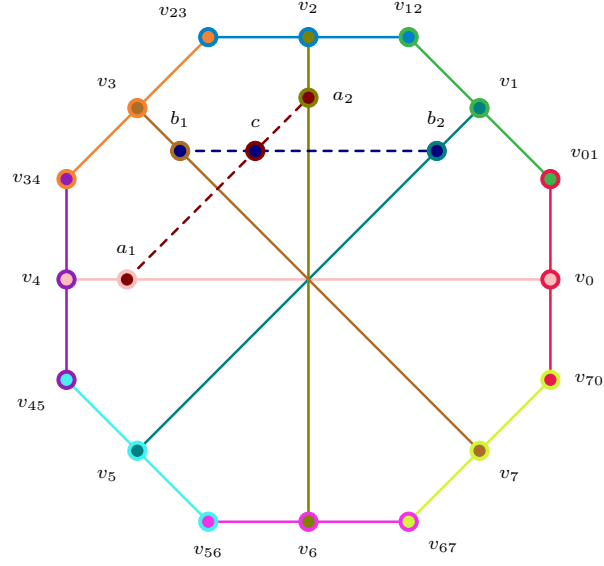
\begin{figure*}[t]
\centering
\begin{tikzpicture}[scale=1.6,every label/.style={font=\scriptsize}]
  \def\Rmid{2}
  \def\Rcor{2.1648}
  \def\Rinn{1.5}
  \foreach \i/\j/\ang in {
    0/1/22.5,1/2/67.5,2/3/112.5,3/4/157.5,
    4/5/202.5,5/6/247.5,6/7/292.5,7/0/337.5
  }{\coordinate (v\i\j) at (\ang:\Rcor);}
  \foreach \i/\ang in {
    0/0,1/45,2/90,3/135,4/180,5/225,6/270,7/315
  }{\coordinate (v\i) at (\ang:\Rmid);}
  \coordinate (a1) at (180:\Rinn);
  \coordinate (a2) at (90:\Rinn);
  \coordinate (b1) at (135:\Rinn);
  \coordinate (b2) at (45:\Rinn);
  \coordinate (c) at ($(a1)!0.70710678!(a2)$);

  \draw[ctxline,draw=ctx1] (v70)--(v01);
  \draw[ctxline,draw=ctx2] (v01)--(v12);
  \draw[ctxline,draw=ctx3] (v12)--(v23);
  \draw[ctxline,draw=ctx4] (v23)--(v34);
  \draw[ctxline,draw=ctx5] (v34)--(v45);
  \draw[ctxline,draw=ctx6] (v45)--(v56);
  \draw[ctxline,draw=ctx7] (v56)--(v67);
  \draw[ctxline,draw=ctx8] (v67)--(v70);
  \draw[ctxline,draw=ctx9] (v0)--(v4);
  \draw[ctxline,draw=ctx10] (v1)--(v5);
  \draw[ctxline,draw=ctx11] (v2)--(v6);
  \draw[ctxline,draw=ctx12] (v3)--(v7);
  \draw[ctxline,dashed,draw=ctx13] (a1)--(c)--(a2);
  \draw[ctxline,dashed,draw=ctx14] (b1)--(c)--(b2);

  \twodotlabel{ctx1}{ctx2}{above right}{$v_{01}$}{(v01)}
  \twodotlabel{ctx2}{ctx3}{above}{$v_{12}$}{(v12)}
  \twodotlabel{ctx3}{ctx4}{above left}{$v_{23}$}{(v23)}
  \twodotlabel{ctx4}{ctx5}{left}{$v_{34}$}{(v34)}
  \twodotlabel{ctx5}{ctx6}{below left}{$v_{45}$}{(v45)}
  \twodotlabel{ctx6}{ctx7}{below}{$v_{56}$}{(v56)}
  \twodotlabel{ctx7}{ctx8}{below right}{$v_{67}$}{(v67)}
  \twodotlabel{ctx8}{ctx1}{right}{$v_{70}$}{(v70)}
  \twodotlabel{ctx1}{ctx9}{right}{$v_0$}{(v0)}
  \twodotlabel{ctx2}{ctx10}{above right}{$v_1$}{(v1)}
  \twodotlabel{ctx3}{ctx11}{above}{$v_2$}{(v2)}
  \twodotlabel{ctx4}{ctx12}{above left}{$v_3$}{(v3)}
  \twodotlabel{ctx5}{ctx9}{left}{$v_4$}{(v4)}
  \twodotlabel{ctx6}{ctx10}{below left}{$v_5$}{(v5)}
  \twodotlabel{ctx7}{ctx11}{below}{$v_6$}{(v6)}
  \twodotlabel{ctx8}{ctx12}{below right}{$v_7$}{(v7)}
  \twodotlabel{ctx9}{ctx13}{above}{$a_1$}{(a1)}
  \twodotlabel{ctx11}{ctx13}{right}{$a_2$}{(a2)}
  \twodotlabel{ctx12}{ctx14}{above}{$b_1$}{(b1)}
  \twodotlabel{ctx10}{ctx14}{above}{$b_2$}{(b2)}
  \twodotlabel{ctx13}{ctx14}{above}{$c$}{(c)}
\end{tikzpicture}
\caption{The 21-ray hypergraph $\Htwentyone$.  The twelve solid contexts
are those of $\Htwenty$; the two dashed contexts meet at the new vertex
$c$.}
\label{fig:h21}
\end{figure*}

\subsection{Exact coordinates over \texorpdfstring{$\C^3$}{C3}}

Let
\begin{equation}
 p=\frac{\sqrt{14}+\sqrt2}{4},\qquad
 q=\frac{\sqrt{14}-\sqrt2}{4}.
 \label{eq:pq}
\end{equation}
Then
\begin{equation}
 p^2+q^2=2,\qquad pq=\frac34.
 \label{eq:pq-identities}
\end{equation}
Choose
\begin{align}
 c&=(0,0,1),\nonumber\\
 a_1&=(1,1,0),&a_2&=(1,-1,0),\nonumber\\
 b_1&=(1,\ii,0),&b_2&=(1,-\ii,0),
 \label{eq:h21-inner}
\end{align}
and the eight midpoint representatives
\begin{align}
 v_0&=(1,-1,\sqrt2),&v_4&=(1,-1,-\sqrt2),\nonumber\\
 v_1&=(1,\ii,p+\ii q),&v_5&=(1,\ii,-p-\ii q),\nonumber\\
 v_2&=(1,1,-\ii\sqrt2),&v_6&=(1,1,\ii\sqrt2),\nonumber\\
 v_3&=(1,-\ii,q-\ii p),&v_7&=(1,-\ii,-q+\ii p).
 \label{eq:h21-midpoints}
\end{align}
The eight corner representatives are
\begin{align}
 v_{01}&=(-p+\ii(q+\sqrt2),\ \sqrt2-p+\ii q,\ 1-\ii),\nonumber\\
 v_{12}&=(\sqrt2-p+\ii q,\ p-\ii(q+\sqrt2),\ 1+\ii),\nonumber\\
 v_{23}&=(q+\sqrt2+\ii p,\ -q+\ii(\sqrt2-p),\ -1+\ii),\nonumber\\
 v_{34}&=(q+\ii(p-\sqrt2),\ q+\sqrt2+\ii p,\ -1-\ii),\nonumber\\
 v_{45}&=(p-\ii(q+\sqrt2),\ -\sqrt2+p-\ii q,\ 1-\ii),\nonumber\\
 v_{56}&=(-\sqrt2+p-\ii q,\ -p+\ii(q+\sqrt2),\ 1+\ii),\nonumber\\
 v_{67}&=(-q-\sqrt2-\ii p,\ q-\ii(\sqrt2-p),\ -1+\ii),\nonumber\\
 v_{70}&=(-q-\ii(p-\sqrt2),\ -q-\sqrt2-\ii p,\ -1-\ii).
 \label{eq:h21-corners}
\end{align}

\begin{theorem}
\label{thm:h21-complex}
Equations~\eqref{eq:pq}--\eqref{eq:h21-corners} give a faithful
orthogonal representation of $\Htwentyone$ in $\C^3$.
\end{theorem}

\begin{proof}
The two added contexts are immediate from Eq.~\eqref{eq:h21-inner}.
For the original inner contexts, representative calculations are
\begin{align*}
 \langle a_1,v_0\rangle&=0,&
 \langle v_0,v_4\rangle&=1+1-2=0,\\
 \langle b_2,v_1\rangle&=0,&
 \langle v_1,v_5\rangle&=2-(p^2+q^2)=0.
\end{align*}
The two remaining inner contexts follow in the same way.

Direct evaluation of the Hermitian cross product gives, exactly,
\[
 v_{j,j+1}=v_j\boxtimes v_{j+1}
\]
for each representative in Eq.~\eqref{eq:h21-corners}.  Hence every
corner is orthogonal to its two incident midpoints.  Substitution in the
Gram closure identity~\eqref{eq:gram-closure} reduces the eight remaining
perimeter inner products to the two relations in
Eq.~\eqref{eq:pq-identities}; they therefore vanish.

Finally, Appendix~\ref{app:faithfulness} gives an exact check of all 210
unordered pairs.  The 42 pairs prescribed by the fourteen contexts have
zero overlap.  Every remaining pair has squared normalized overlap
either $1/4$ or $1/2$.  Thus there are no additional orthogonalities and
no ray coincidences.
\end{proof}

\subsection{Nonexistence over \texorpdfstring{$\R^3$}{R3}}

\begin{theorem}
\label{thm:h21-real}
The hypergraph $\Htwentyone$ has no faithful orthogonal representation
in $\R^3$.
\end{theorem}

\begin{proof}
Assume that a faithful real representation exists.  A global orthogonal
transformation and independent rescalings of ray representatives allow
us to set
\begin{equation}
 c=(0,0,1),\qquad a_1=(1,0,0),\qquad a_2=(0,1,0).
 \label{eq:real-normal-a}
\end{equation}
The second added context lies in the plane $c^\perp$, so for some angle
$\theta$,
\begin{equation}
 b_1=(\cos\theta,\sin\theta,0),\qquad
 b_2=(-\sin\theta,\cos\theta,0).
 \label{eq:real-normal-b}
\end{equation}
Faithfulness excludes $\theta\in\frac{\pi}{2}\mathbb Z$, since those
values identify a $b$-ray with an $a$-ray.  Put
\begin{equation}
 m=\tan\theta\ne0,\qquad M=1+m^2.
 \label{eq:mM}
\end{equation}

Write a representative vector of a midpoint ray as
$v_j=(X_j,Y_j,Z_j)$.  No midpoint ray can lie in $c^\perp$.
Indeed, a planar vector orthogonal to the inner ray in its
context is proportional to the other member of one of the two planar
bases $\{a_1,a_2\}$ or $\{b_1,b_2\}$, contradicting faithfulness.
Thus $Z_j\ne0$ for every $j$, and each representative may be rescaled to
\begin{equation}
 v_j=(X_j,Y_j,1).
 \label{eq:affine-midpoints}
\end{equation}

The inner contexts $D_0,D_1,D_3$ imply
\begin{align}
 X_0&=X_4=0,&Y_0Y_4+1&=0,
 \label{eq:D0-real}\\
 Y_1&=mX_1,&Y_5&=mX_5,&MX_1X_5+1&=0,
 \label{eq:D1-real}\\
 X_3&=-mY_3,&X_7&=-mY_7,&MY_3Y_7+1&=0.
 \label{eq:D3-real}
\end{align}
In particular,
\begin{equation}
 X_5=-\frac1{MX_1},\qquad
 Y_7=-\frac1{MY_3},
 \label{eq:X5Y7}
\end{equation}
and $X_1Y_3\ne0$.

The corner rays can be eliminated from the perimeter by the real form
of Lemma~\ref{lem:gram-closure}.  At $j=0$,
Eq.~\eqref{eq:gram-closure} factors as
\begin{equation}
 Y_0\bigl((mX_1Y_7-1)Y_0+Y_7+mX_1\bigr)=0.
 \label{eq:j0-factor}
\end{equation}
The first factor cannot vanish because of Eq.~\eqref{eq:D0-real}.
Moreover, the coefficient $1-mX_1Y_7$ cannot vanish: otherwise
Eq.~\eqref{eq:j0-factor} would also give $Y_7+mX_1=0$, and the two
relations would imply $-m^2X_1^2=1$.  Therefore
\begin{equation}
 Y_0=\frac{Y_7+mX_1}{1-mX_1Y_7}.
 \label{eq:Y0-first}
\end{equation}
The closure equation at $j=4$ gives analogously
\begin{equation}
 Y_4=\frac{Y_3+mX_5}{1-mY_3X_5},
 \label{eq:Y4-first}
\end{equation}
with a nonzero denominator.

Substituting Eq.~\eqref{eq:X5Y7} into
Eqs.~\eqref{eq:Y0-first} and~\eqref{eq:Y4-first} yields
\begin{equation}
 Y_0=\frac{mMX_1Y_3-1}{MY_3+mX_1},\qquad
 Y_4=\frac{MX_1Y_3-m}{MX_1+mY_3}.
 \label{eq:Y0Y4-final}
\end{equation}
Set $u=X_1Y_3$.  The remaining equation $Y_0Y_4=-1$ becomes
\begin{equation}
 (mMu-1)(Mu-m)=-(mX_1+MY_3)(MX_1+mY_3).
 \label{eq:cross-multiplied}
\end{equation}
Expanding the left-hand side gives
\begin{align*}
 (mMu-1)(Mu-m)
 &=mM^2u^2-M(m^2+1)u+m\\
 &=mM^2u^2-M^2u+m,
\end{align*}
where the second equality uses $M=1+m^2$.
The right-hand side is
\[
 -u(M^2+m^2)-mM(X_1^2+Y_3^2).
\]
After cancellation and division by $m\ne0$, Eq.~\eqref{eq:cross-multiplied}
is therefore equivalent to the master equation
\begin{equation}
 M^2u^2+mu+1+M(X_1^2+Y_3^2)=0.
 \label{eq:master}
\end{equation}

Since $X_1,Y_3$ are real,
$X_1^2+Y_3^2\ge2|X_1Y_3|=2|u|$.  The left-hand side of
Eq.~\eqref{eq:master} is consequently at least
\begin{align*}
 M^2u^2+mu+1+2M|u|
 &\ge 1+M^2u^2+|u|(2M-|m|)\\
 &\ge1.
\end{align*}
Here, with $t=|m|$,
\[
 2M-|m|=2t^2-t+2=2\left(t-\frac14\right)^2+\frac{15}{8}>0.
\]
This contradicts
Eq.~\eqref{eq:master}, and the assumed real representation cannot exist.
\end{proof}

\section{Faithful realification in six dimensions}
\label{sec:realification}

Theorems~\ref{thm:h20-real} and~\ref{thm:h21-real} keep the ambient
dimension fixed.  When dimension doubling is allowed, both hypergraphs have
faithful real representations.  We first recall the finite construction and
then give explicit phase choices for the present
coordinates~\cite{svozil-2025-rvc}.

\begin{theorem}[Phase-adjusted realification]
\label{thm:realification}
Every finite faithful orthogonal representation by rays in $\C^3$ induces a
faithful orthogonal representation of the same hypergraph by rays in $\R^6$.
\end{theorem}

\begin{proof}
Let $\psi_1,\ldots,\psi_N\in\C^3$ be nonzero representatives and put
$c_{k\ell}=\langle\psi_k,\psi_\ell\rangle$.  Define the canonical
realification
\begin{equation}
 \Phi_0(z_1,z_2,z_3)
 = (\Re z_1,\Re z_2,\Re z_3,\Im z_1,\Im z_2,\Im z_3).
 \label{eq:canonical-realification}
\end{equation}
For phases $\theta_k\in\R$, set
$R_k=\Phi_0(e^{\ii\theta_k}\psi_k)$.  Direct expansion gives
\begin{equation}
 R_k\mathbin{\cdot}R_\ell
 =\Re\!\left(e^{\ii(\theta_\ell-\theta_k)}c_{k\ell}\right).
 \label{eq:realified-dot}
\end{equation}
Thus $c_{k\ell}=0$ implies $R_k\cdot R_\ell=0$ for every phase choice.
If $c_{k\ell}\ne0$, write
$c_{k\ell}=|c_{k\ell}|e^{\ii\varphi_{k\ell}}$.  The right-hand side of
Eq.~\eqref{eq:realified-dot} vanishes precisely when
\begin{equation}
 \theta_\ell-\theta_k+\varphi_{k\ell}
 \equiv\frac{\pi}{2}\pmod\pi.
 \label{eq:forbidden-phase}
\end{equation}

Choose $\theta_1,\ldots,\theta_N$ inductively.  Suppose that the first $m-1$
phases have been fixed.  For each $k<m$ with $c_{km}\ne0$, the unwanted
orthogonality equation $R_k\mathbin{\cdot}R_m=0$ excludes, by
Eq.~\eqref{eq:forbidden-phase}, exactly two values of $\theta_m$ modulo
$2\pi$.  Their union is finite, so a phase outside it exists.  Continuing
through $m=N$ makes the real dot product nonzero for every originally
nonorthogonal pair, while all original orthogonalities remain zero.

It remains only to check ray distinctness.  If $R_k$ and $R_\ell$ were real
proportional, then $e^{\ii\theta_k}\psi_k$ and
$e^{\ii\theta_\ell}\psi_\ell$ would be real proportional and hence would
represent the same complex ray.  Faithfulness of the original representation
excludes this.  Therefore the realified representation is faithful.
\end{proof}

For explicit coordinates it is convenient to parametrize the phase by a real
number $t$:
\begin{equation}
 \lambda_t=\frac{1+\ii t}{\sqrt{1+t^2}},\qquad
 \Phi_t(z)=\Phi_0(\lambda_tz).
 \label{eq:lambda-t}
\end{equation}
Writing $z=x+\ii y$ with $x,y\in\R^3$ gives the entirely real formula
\begin{equation}
 \Phi_t(x+\ii y)=\frac{(x-ty,\ y+tx)}{\sqrt{1+t^2}}\in\R^6.
 \label{eq:explicit-realification}
\end{equation}
Moreover,
\begin{equation}
 \Phi_{t_u}(u)\mathbin{\cdot}\Phi_{t_v}(v)
 =\Re\!\left(
 \frac{(1-\ii t_u)(1+\ii t_v)}
 {\sqrt{(1+t_u^2)(1+t_v^2)}}\langle u,v\rangle
 \right).
 \label{eq:t-dot}
\end{equation}

\begin{corollary}[Explicit real six-dimensional coordinates]
\label{cor:r6}
Let $\psi_v^{(20)}$ denote the representatives in
Eqs.~\eqref{eq:C-definition}--\eqref{eq:h20-corner-formula}.  Define
\begin{equation}
 t_v^{(20)}=
 \begin{cases}
  1,&v\in\{v_{34},v_{45},v_{56}\},\\
  2,&v\in\{v_{67},v_{70}\},\\
  0,&\text{otherwise}.
 \end{cases}
 \label{eq:t20}
\end{equation}
Then $R_v^{(20)}=\Phi_{t_v^{(20)}}(\psi_v^{(20)})$ is a faithful orthogonal
representation of $\Htwenty$ in $\R^6$.

Likewise, let $\psi_v^{(21)}$ denote the representatives in
Eqs.~\eqref{eq:pq}--\eqref{eq:h21-corners}, and define
\begin{equation}
 t_v^{(21)}=
 \begin{cases}
  1,&v=v_{45},\\
  2,&v\in\{v_2,v_3,v_6,v_7,v_{34},v_{56},v_{70}\},\\
  3,&v=v_{67},\\
  0,&\text{otherwise}.
 \end{cases}
 \label{eq:t21}
\end{equation}
Then $R_v^{(21)}=\Phi_{t_v^{(21)}}(\psi_v^{(21)})$ is a faithful orthogonal
representation of $\Htwentyone$ in $\R^6$.
\end{corollary}

\begin{proof}
Equations~\eqref{eq:explicit-realification}, \eqref{eq:t20}, and
\eqref{eq:t21} give all six real coordinates directly from the complex
vectors already displayed.  Exact substitution in Eq.~\eqref{eq:t-dot}
gives zero for exactly 36 of the $\binom{20}{2}=190$ pairs in $\Htwenty$,
namely the pairs prescribed by its twelve contexts.  The other 154 real dot
products are nonzero.  For $\Htwentyone$, exactly 42 of the
$\binom{21}{2}=210$ real dot products vanish, namely the pairs prescribed by
its fourteen contexts, and the other 168 are nonzero.  These calculations
take place in the algebraic fields generated by the radicals in the complex
coordinates and by $\sqrt{1+t^2}$; no numerical tolerance is involved.
Ray distinctness follows from Theorem~\ref{thm:realification}.
\end{proof}

Each displayed three-ray context is therefore an orthogonal triple in
$\R^6$, but it spans only a three-dimensional subspace.  Its projector sum
is the rank-three projector onto that subspace, not $I_6$.  This loss of
maximality is precisely why the real six-dimensional representations do not
contradict the two real three-dimensional no-go theorems.

\section{Discussion}

The two nonexistence proofs expose different real-complex mechanisms.
For $\Htwenty$, the context covering forces a rank-two positive operator
to have two disjoint, nonorthogonal decompositions into two rank-one
projectors.  Complex relative phases make this possible.  Over the
reals, the two simple eigenvalues determine the two rays, producing the
rigidity of Lemma~\ref{lem:real-rigidity}.

For $\Htwentyone$, both inner pairs are orthogonal and their common
projector sum is the identity on $c^\perp$.  The eigenvalue argument is
therefore unavailable: the relevant eigenvalue is degenerate and many
real orthogonal bases of $c^\perp$ exist.  The obstruction instead comes
from compatibility with the octagonal perimeter.  The two reduced
closure equations determine $Y_0$ and $Y_4$, while the inner context
$D_0$ demands $Y_0Y_4=-1$; together they produce the positive master
expression in Eq.~\eqref{eq:master}.

Both statements are dimension-specific.  Corollary~\ref{cor:r6} realizes the
same pairwise orthogonality relations faithfully in $\R^6$, but the specified
triples are no longer complete bases.  Consequently, the three-dimensional
context equation~\eqref{eq:context-sum} is replaced by a rank-three subspace
projector.  The complementary covering no longer forces
Eq.~\eqref{eq:covering-identity}, and the three-dimensional cross-product
closure used for $\Htwentyone$ no longer applies.  The comparison therefore
concerns the scalar field at fixed dimension and fixed context maximality,
not the possibility of simulating complex coordinates in a larger real
space.  It is also independent of KS noncolorability: the contradictions use
orthogonality, completeness of contexts, and field-dependent algebra, not an
impossibility of assigning two-valued states.

These fixed-dimension witnesses may also guide the design of protocols in
which a real implementation must either enlarge the Hilbert-space dimension
or relinquish maximality of the specified contexts.

Finally, faithfulness is essential.  If unprescribed orthogonalities or
ray coincidences were allowed, the covering identity could collapse to
a degenerate realization.  Conversely, deleting contexts weakens the
projector identity or the perimeter closure system.  The relevant
object is therefore the complete pair consisting of the vertex set and
its specified contexts.

\section{Conclusion}

The hypergraphs $\Htwenty$ and $\Htwentyone$ give two exact finite
separations between complex and real three-dimensional orthogonality
representations.  The first is governed by the field dependence of a
two-projector decomposition; the second survives at the orthogonal limit
and is governed by cyclic closure.  Explicit algebraic vectors,
complete nonexistence proofs, and exact pairwise-overlap certificates
establish both results without numerical or coloring assumptions.  Explicit
phase-adjusted coordinates in $\R^6$ show at the same time that the
obstruction disappears after dimension doubling, once the original
three-ray contexts are no longer required to be maximal.

\appendix

\section{Exact faithfulness certificates}
\label{app:faithfulness}

For two nonzero representatives $x,y$, define their squared normalized
overlap by
\begin{equation}
 Q(x,y)=\frac{|\langle x,y\rangle|^2}
 {\langle x,x\rangle\langle y,y\rangle}.
 \label{eq:Q}
\end{equation}
Thus $Q=0$ exactly for orthogonal rays, while $Q=1$ exactly for
coincident rays.  All calculations below are exact reductions in the
algebraic fields generated by the displayed radicals and phases.  The
reductions and pair counts were independently cross-checked by exact symbolic
arithmetic; no floating-point tolerance is used.

For $\Htwenty$, the 12 contexts prescribe 36 distinct unordered
orthogonal pairs.  Substitution of
Eqs.~\eqref{eq:C-definition}--\eqref{eq:h20-corner-formula}, followed by
reduction with Eq.~\eqref{eq:C-polynomial}, gives $Q=0$ for precisely
those 36 pairs.  For the other 154 pairs, the complete list of values
and multiplicities is:
\begin{equation}
\begin{array}{c@{\;}r@{\quad}c@{\;}r@{\quad}c@{\;}r}
 Q&\text{mult.}&Q&\text{mult.}&Q&\text{mult.}\\ \hline
 \frac1{25}&2&\frac3{10}&8&\frac8{15}&8\\
 \frac3{50}&4&\frac{23}{75}&2&\frac{14}{25}&8\\
 \frac7{75}&8&\frac{49}{150}&8&\frac{16}{25}&2\\
 \frac8{75}&8&\frac{26}{75}&8&\frac7{10}&8\\
 \frac7{50}&8&\frac{28}{75}&8&\frac{59}{75}&2\\
 \frac4{25}&8&\frac{32}{75}&4&\frac45&16\\
 \frac15&16&\frac7{15}&8&&\\[-2pt]
 \frac6{25}&8&\frac{13}{25}&2&&
\end{array}
\label{eq:H20-Q-table}
\end{equation}
The multiplicities sum to 154.  In particular,
\begin{equation}
 \frac1{25}\le Q(x,y)\le\frac45
 \label{eq:H20-Q-bounds}
\end{equation}
for every unprescribed pair.  This proves simultaneously that no such
pair is orthogonal and that no two vertices label the same ray.

For $\Htwentyone$, the 14 contexts prescribe 42 distinct unordered
orthogonal pairs.  Exact substitution of
Eqs.~\eqref{eq:pq-identities}--\eqref{eq:h21-corners} gives zero for
precisely those pairs.  The other 168 pairs split as
\begin{equation}
 \begin{array}{c|cc}
 Q&\frac14&\frac12\\ \hline
 \text{multiplicity}&84&84
 \end{array}.
 \label{eq:H21-Q-table}
\end{equation}
Again every unprescribed overlap lies strictly between zero and one.
This completes the exact faithfulness verification used in
Theorems~\ref{thm:h20-complex} and~\ref{thm:h21-complex}.

\begin{acknowledgments}
During manuscript preparation, the authors used OpenAI Codex (GPT-5) to assist with literature organization, LaTeX editing, exact symbolic checks, and figure/PDF verification; all AI-assisted output was directed and critically reviewed by the authors, who independently verified the mathematics and take full responsibility for the final content, and no AI tool is an author.
This research was funded in part by the Austrian
Science Fund (FWF), Grant DOI \href{https://doi.org/10.55776/PIN5424624}{10.55776/PIN5424624}, and the Czech Science
Foundation (GA\v{C}R), Grant No.~25-20013L.  The authors declare no
conflict of interest.
\end{acknowledgments}

\bibliography{svozil.bib}

\begin{thebibliography}{12}%
\makeatletter
\providecommand \@ifxundefined [1]{%
 \@ifx{#1\undefined}
}%
\providecommand \@ifnum [1]{%
 \ifnum #1\expandafter \@firstoftwo
 \else \expandafter \@secondoftwo
 \fi
}%
\providecommand \@ifx [1]{%
 \ifx #1\expandafter \@firstoftwo
 \else \expandafter \@secondoftwo
 \fi
}%
\providecommand \natexlab [1]{#1}%
\providecommand \enquote  [1]{``#1''}%
\providecommand \bibnamefont  [1]{#1}%
\providecommand \bibfnamefont [1]{#1}%
\providecommand \citenamefont [1]{#1}%
\providecommand \href@noop [0]{\@secondoftwo}%
\providecommand \href [0]{\begingroup \@sanitize@url \@href}%
\providecommand \@href[1]{\@@startlink{#1}\@@href}%
\providecommand \@@href[1]{\endgroup#1\@@endlink}%
\providecommand \@sanitize@url [0]{\catcode `\\12\catcode `\$12\catcode
  `\&12\catcode `\#12\catcode `\^12\catcode `\_12\catcode `\%12\relax}%
\providecommand \@@startlink[1]{}%
\providecommand \@@endlink[0]{}%
\providecommand \url  [0]{\begingroup\@sanitize@url \@url }%
\providecommand \@url [1]{\endgroup\@href {#1}{\urlprefix }}%
\providecommand \urlprefix  [0]{URL }%
\providecommand \Eprint [0]{\href }%
\providecommand \doibase [0]{https://doi.org/}%
\providecommand \selectlanguage [0]{\@gobble}%
\providecommand \bibinfo  [0]{\@secondoftwo}%
\providecommand \bibfield  [0]{\@secondoftwo}%
\providecommand \translation [1]{[#1]}%
\providecommand \BibitemOpen [0]{}%
\providecommand \bibitemStop [0]{}%
\providecommand \bibitemNoStop [0]{.\EOS\space}%
\providecommand \EOS [0]{\spacefactor3000\relax}%
\providecommand \BibitemShut  [1]{\csname bibitem#1\endcsname}%
\let\auto@bib@innerbib\@empty
\bibitem [{\citenamefont {Harding}\ and\ \citenamefont
  {Salinas~Schmeis}(2025)}]{harding2025remarksIJTP}%
  \BibitemOpen
  \bibfield  {author} {\bibinfo {author} {\bibfnamefont {J.}~\bibnamefont
  {Harding}}\ and\ \bibinfo {author} {\bibfnamefont {R.}~\bibnamefont
  {Salinas~Schmeis}},\ }\bibfield  {title} {\bibinfo {title} {Remarks on
  orthogonality spaces},\ }\href
  {https://doi.org/10.48550/10.1007/s10773-025-06062-x} {\bibfield  {journal}
  {\bibinfo  {journal} {International Journal of Theoretical Physics}\ }\textbf
  {\bibinfo {volume} {64}},\ \bibinfo {pages} {211} (\bibinfo {year} {2025})},\
  \Eprint {https://arxiv.org/abs/2505.13871} {arXiv:2505.13871 [math-ph]}
  \BibitemShut {NoStop}%
\bibitem [{\citenamefont {Trandafir}\ and\ \citenamefont
  {Cabello}(2025)}]{2024-cabello-trandafir0}%
  \BibitemOpen
  \bibfield  {author} {\bibinfo {author} {\bibfnamefont {S.}~\bibnamefont
  {Trandafir}}\ and\ \bibinfo {author} {\bibfnamefont {A.}~\bibnamefont
  {Cabello}},\ }\href {https://doi.org/10.1103/PhysRevA.111.022408} {\bibinfo
  {title} {Optimal conversion of {K}ochen-{S}pecker sets into bipartite perfect
  quantum strategies}} (\bibinfo {year} {2025}),\ \Eprint
  {https://arxiv.org/abs/arXiv:2410.17470} {arXiv:2410.17470} \BibitemShut
  {NoStop}%
\bibitem [{\citenamefont {Navara}\ and\ \citenamefont
  {Svozil}(2025)}]{svozil-2025-MYOCHS}%
  \BibitemOpen
  \bibfield  {author} {\bibinfo {author} {\bibfnamefont {M.}~\bibnamefont
  {Navara}}\ and\ \bibinfo {author} {\bibfnamefont {K.}~\bibnamefont
  {Svozil}},\ }\href {https://doi.org/10.48550/arXiv.2509.08636} {\bibinfo
  {title} {Construction of {K}ochen-{S}pecker sets from mutually unbiased
  bases}} (\bibinfo {year} {2025}),\ \Eprint {https://arxiv.org/abs/2509.08636}
  {arXiv:2509.08636 [quant-ph]} \BibitemShut {NoStop}%
\bibitem [{\citenamefont {Hardy}\ and\ \citenamefont
  {Wootters}(2012)}]{Hardy2012LimitedHolism}%
  \BibitemOpen
  \bibfield  {author} {\bibinfo {author} {\bibfnamefont {L.}~\bibnamefont
  {Hardy}}\ and\ \bibinfo {author} {\bibfnamefont {W.~K.}\ \bibnamefont
  {Wootters}},\ }\bibfield  {title} {\bibinfo {title} {Limited holism and
  real-vector-space quantum theory},\ }\href
  {https://doi.org/10.1007/s10701-011-9616-6} {\bibfield  {journal} {\bibinfo
  {journal} {Foundations of Physics}\ }\textbf {\bibinfo {volume} {42}},\
  \bibinfo {pages} {454} (\bibinfo {year} {2012})}\BibitemShut {NoStop}%
\bibitem [{\citenamefont {Aleksandrova}\ \emph {et~al.}(2013)\citenamefont
  {Aleksandrova}, \citenamefont {Borish},\ and\ \citenamefont
  {Wootters}}]{Aleksandrova2013UniversalQubit}%
  \BibitemOpen
  \bibfield  {author} {\bibinfo {author} {\bibfnamefont {A.}~\bibnamefont
  {Aleksandrova}}, \bibinfo {author} {\bibfnamefont {V.}~\bibnamefont
  {Borish}},\ and\ \bibinfo {author} {\bibfnamefont {W.~K.}\ \bibnamefont
  {Wootters}},\ }\bibfield  {title} {\bibinfo {title} {Real-vector-space
  quantum theory with a universal quantum bit},\ }\href
  {https://doi.org/10.1103/PhysRevA.87.052106} {\bibfield  {journal} {\bibinfo
  {journal} {Physical Review A}\ }\textbf {\bibinfo {volume} {87}},\ \bibinfo
  {pages} {052106} (\bibinfo {year} {2013})}\BibitemShut {NoStop}%
\bibitem [{\citenamefont {McKague}\ \emph {et~al.}(2009)\citenamefont
  {McKague}, \citenamefont {Mosca},\ and\ \citenamefont
  {Gisin}}]{McKague2009SimulatingRealHilbert}%
  \BibitemOpen
  \bibfield  {author} {\bibinfo {author} {\bibfnamefont {M.}~\bibnamefont
  {McKague}}, \bibinfo {author} {\bibfnamefont {M.}~\bibnamefont {Mosca}},\
  and\ \bibinfo {author} {\bibfnamefont {N.}~\bibnamefont {Gisin}},\ }\bibfield
   {title} {\bibinfo {title} {Simulating quantum systems using real hilbert
  spaces},\ }\href {https://doi.org/10.1103/PhysRevLett.102.020505} {\bibfield
  {journal} {\bibinfo  {journal} {Physical Review Letters}\ }\textbf {\bibinfo
  {volume} {102}},\ \bibinfo {pages} {020505} (\bibinfo {year}
  {2009})}\BibitemShut {NoStop}%
\bibitem [{\citenamefont {Renou}\ \emph {et~al.}(2021)\citenamefont {Renou},
  \citenamefont {Trillo}, \citenamefont {Weilenmann}, \citenamefont {Le},
  \citenamefont {Tavakoli}, \citenamefont {Gisin}, \citenamefont {Ac\'in},\
  and\ \citenamefont {Navascu\'es}}]{renou-2021}%
  \BibitemOpen
  \bibfield  {author} {\bibinfo {author} {\bibfnamefont {M.-O.}\ \bibnamefont
  {Renou}}, \bibinfo {author} {\bibfnamefont {D.}~\bibnamefont {Trillo}},
  \bibinfo {author} {\bibfnamefont {M.}~\bibnamefont {Weilenmann}}, \bibinfo
  {author} {\bibfnamefont {T.~P.}\ \bibnamefont {Le}}, \bibinfo {author}
  {\bibfnamefont {A.}~\bibnamefont {Tavakoli}}, \bibinfo {author}
  {\bibfnamefont {N.}~\bibnamefont {Gisin}}, \bibinfo {author} {\bibfnamefont
  {A.}~\bibnamefont {Ac\'in}},\ and\ \bibinfo {author} {\bibfnamefont
  {M.}~\bibnamefont {Navascu\'es}},\ }\bibfield  {title} {\bibinfo {title}
  {Quantum theory based on real numbers can be experimentally falsified},\
  }\href {https://doi.org/10.1038/s41586-021-04160-4} {\bibfield  {journal}
  {\bibinfo  {journal} {Nature}\ }\textbf {\bibinfo {volume} {600}},\ \bibinfo
  {pages} {625} (\bibinfo {year} {2021})},\ \Eprint
  {https://arxiv.org/abs/arXiv:2101.10873} {arXiv:2101.10873} \BibitemShut
  {NoStop}%
\bibitem [{\citenamefont {Weilenmann}\ \emph {et~al.}(2025)\citenamefont
  {Weilenmann}, \citenamefont {Gisin},\ and\ \citenamefont
  {Sekatski}}]{Weilenmann2025PartialIndependencePRL}%
  \BibitemOpen
  \bibfield  {author} {\bibinfo {author} {\bibfnamefont {M.}~\bibnamefont
  {Weilenmann}}, \bibinfo {author} {\bibfnamefont {N.}~\bibnamefont {Gisin}},\
  and\ \bibinfo {author} {\bibfnamefont {P.}~\bibnamefont {Sekatski}},\
  }\bibfield  {title} {\bibinfo {title} {Partial independence suffices to rule
  out real quantum theory experimentally},\ }\href
  {https://doi.org/10.1103/3fv7-p8cs} {\bibfield  {journal} {\bibinfo
  {journal} {Physical Review Letters}\ }\textbf {\bibinfo {volume} {135}},\
  \bibinfo {pages} {180201} (\bibinfo {year} {2025})}\BibitemShut {NoStop}%
\bibitem [{\citenamefont {Bednorz}\ and\ \citenamefont
  {Batle}(2022)}]{Bednorz2022OptimalDiscrimination}%
  \BibitemOpen
  \bibfield  {author} {\bibinfo {author} {\bibfnamefont {A.}~\bibnamefont
  {Bednorz}}\ and\ \bibinfo {author} {\bibfnamefont {J.}~\bibnamefont
  {Batle}},\ }\bibfield  {title} {\bibinfo {title} {Optimal discrimination
  between real and complex quantum theories},\ }\href
  {https://doi.org/10.1103/PhysRevA.106.042207} {\bibfield  {journal} {\bibinfo
   {journal} {Physical Review A}\ }\textbf {\bibinfo {volume} {106}},\ \bibinfo
  {pages} {042207} (\bibinfo {year} {2022})}\BibitemShut {NoStop}%
\bibitem [{\citenamefont {Khrennikov}\ and\ \citenamefont
  {Svozil}(2026)}]{svozil-2025-rvc}%
  \BibitemOpen
  \bibfield  {author} {\bibinfo {author} {\bibfnamefont {A.}~\bibnamefont
  {Khrennikov}}\ and\ \bibinfo {author} {\bibfnamefont {K.}~\bibnamefont
  {Svozil}},\ }\bibfield  {title} {\bibinfo {title} {Faithful real embedding of
  a three-dimensional complex kochen-specker configuration},\ }\href
  {https://doi.org/10.1103/2pf4-bq7p} {\bibfield  {journal} {\bibinfo
  {journal} {Physical Review A}\ }\textbf {\bibinfo {volume} {113}},\ \bibinfo
  {pages} {032206} (\bibinfo {year} {2026})},\ \Eprint
  {https://arxiv.org/abs/arXiv:2511.17223} {arXiv:2511.17223} \BibitemShut
  {NoStop}%
\bibitem [{\citenamefont {Lov\'asz}(1979)}]{lovasz-79}%
  \BibitemOpen
  \bibfield  {author} {\bibinfo {author} {\bibfnamefont {L.}~\bibnamefont
  {Lov\'asz}},\ }\bibfield  {title} {\bibinfo {title} {On the {S}hannon
  capacity of a graph},\ }\href {https://doi.org/10.1109/TIT.1979.1055985}
  {\bibfield  {journal} {\bibinfo  {journal} {IEEE Transactions on Information
  Theory}\ }\textbf {\bibinfo {volume} {25}},\ \bibinfo {pages} {1} (\bibinfo
  {year} {1979})}\BibitemShut {NoStop}%
\bibitem [{\citenamefont {Sol\'is-Encina}\ and\ \citenamefont
  {Portillo}(2015)}]{Portillo-2015}%
  \BibitemOpen
  \bibfield  {author} {\bibinfo {author} {\bibfnamefont {A.}~\bibnamefont
  {Sol\'is-Encina}}\ and\ \bibinfo {author} {\bibfnamefont {J.~R.}\
  \bibnamefont {Portillo}},\ }\href {https://doi.org/10.48550/arXiv.1504.03662}
  {\bibinfo {title} {Orthogonal representation of graphs}} (\bibinfo {year}
  {2015}),\ \Eprint {https://arxiv.org/abs/arXiv:1504.03662} {arXiv:1504.03662}
  \BibitemShut {NoStop}%
\end{thebibliography}%

\end{document}